\documentclass[12pt]{amsart}
\usepackage[colorlinks=true]{hyperref}
\usepackage{amsrefs}

\newtheorem{thm}{Theorem}
\newtheorem{prop}[thm]{Proposition}
\newtheorem{lem}[thm]{Lemma}

\newtheorem{rem}[thm]{Remark}
\newtheorem{exmp}[thm]{Example}
\DeclareMathOperator{\sr}{sr}
\DeclareMathOperator{\M}{M}
\title[Kaplansky's condition and almost stable range $1$]{The Kaplansky condition and rings of almost stable range $1$}
\author[Moshe Roitman]{
Moshe Roitman\\
Department of Mathematics\\
University of Haifa\\
Mount Carmel, Haifa 31905, Israel}
\thanks{Part of this work was done while visiting the New Mexico State University. I thank Bruce Olberding from this university for useful discussions and suggestions concerning this topic}
\subjclass[2010]{Primary 13F99}
\keywords{almost stable range $1$, elementary divisor ring, Kaplansky condition, K-Hermite, stable range}
\begin{document}
\begin{abstract}
We  present some variants of the Kaplansky condition for a K-Hermite ring $R$ to be an elementary divisor ring; for example, a commutative K-Hermite ring $R$ is an EDR iff for any elements $x,y,z\in R$ such that  $(x,y)=(1)$, there exists an element $\lambda\in R$ such that $x+\lambda y=uv$, where
$(u,z)=(v,1-z)=(1)$.

We present an example of a a B\'ezout domain that is an elementary divisor ring, but it does not have almost stable range $1$, thus answering a question of
Warren Wm. McGovern.
\end{abstract}
\maketitle

\section{Introduction}
First we recall some basic  definitions and known results.

All rings here are commutative with  unity. A ring $R$ is {\em B\'ezout} if each finitely generated ideal of $R$ is principal.

Two rectangular matrices $A$ and $B$ in $\M_{m,n}(R)$ are {\em equivalent} if there exist invertible matrices $P\in\M_{m,m}(R)$ and $Q\in\M_{n,n}(R)$ such that $B=PAQ$.

The ring $R$ is {\em K-Hermite} if  every rectangular matrix $A$  over $R$ is equivalent to an upper or a lower triangular matrix (following \cite[Appendix to \S 4]{Lam-Serrprobprojmodu:06} we use the term `K-Hermite' rather then `Hermite' as in \cite{Kaplansky-Elemdivimodu:49}).
From \cite{Kaplansky-Elemdivimodu:49} it follows that this definition is equivalent to the definition given there. See also
\cite[Theorem 3]{Henriksen-Someremaelemdivi:55}: by this theorem, a ring is K-Hermite iff for every two comaximal elements $a,b\in R$, there are two comaximal elements $a_1,b_1\in R$ such that $(a,b)=(a_1a+b_1b)$ in $R$. Parentheses $()$ are used to denote the ideal generated by the specified elements.

A ring $R$ is an {\em elementary divisor  ring (EDR)} iff  every rectangular matrix $A$  over $R$ is equivalent to a diagonal matrix.
It follows from \cite{Kaplansky-Elemdivimodu:49} that this definition is equivalent to the definition given there.

An EDR is K-Hermite and a K-Hermite ring is B\'ezout. An integral domain is B\'ezout iff it is K-Hermite.

By \cite[Theorem 6]{Gillman+Henriksen-Someremaabouelem:56} a ring $R$ is an EDR iff it satisfies the following two conditions:

\begin{enumerate}
\item
$R$ is K-Hermite;

\item
$R$ satisfies Kaplansky's condition (see \S2, condition (K) below).
\end{enumerate}

By \cite[Example 4.11]{Gillman+Henriksen-Ringcontfuncwhic:56}, $(1)\ \not\hspace{-.3cm}\implies (2)$. The question in
\cite{Henriksen-Someremaelemdivi:55} whether a B\'ezout domain is an EDR (equivalently, whether it  satisfies Kaplansky's condition), is still open.
On the other hand, it is immediate that $(2)\hspace{.4cm} \not\hspace{-.3cm}\implies (1)$ (Remark \ref{kaplsr1}  below).

In section \ref{2}, we elaborate on the Kaplansky condition.

A row $[r_1,\dots,r_n]$ over a ring $R$ is {\em unimodular} if the elements $r_1,\dots,r_n$ generate the ideal $R$.
The {\em stable range}  $\sr(R)$ of a ring  $R$ is the least integer $n\ge1$ (if it exists) such that  for any unimodular row $[r_1, . . . , r_{n+1}]$ over $R$, there exist
$t_1, . . . , t_n\in R$ such that the row $$[r_1 + t_1r_{n+1}, . . . , r_n+t_nr_{n+1}]$$ is unimodular (see comments on \cite[Theorem 5.2, Ch. VIII]{Lam-Serrprobprojmodu:06}).
For background on stable range
see \cite[\S 3, Ch. V]{Bass-Alge:68}.

The ring $R$ has {\em almost stable range $1$} if every proper homomorphic image of $R$ has stable range $1$ (see \cite{McGovern-Bezoringwithalmo:08}). By \cite[Theorem 3.7]{McGovern-Bezoringwithalmo:08} a B\'ezout ring with almost stable range $1$ is an EDR.
We elaborate on the almost stable range $1$ condition in \S 3.
In particular, we present an elementary divisor domain (so B\'ezout) that does not have almost stable range $1$, thus answering  the question of Warren Wm. McGovern in \cite{McGovern-Bezoringwithalmo:08} (Example \ref{warrenquest}).
  By Remark  \ref{almoststable} below, a ring of stable range $1$ is of almost stable range $1$. On the other hand, $\mathbb Z$ is of almost stable range $1$, but is not of stable range $1$ (the stable range of  $\mathbb Z$ is $2$. More generally, the stable range of any K-Hermite ring is $\le 2$ \cite[Proposition 8]{Menal+Moncasi-reguringwithstab:82}. Also by \cite[Theorem 1]{-RedumatroverBezo:03}, a B\'ezout ring is K-Hermite iff is of stable range $\le2$).

For general background see \cite{Kaplansky-Elemdivimodu:49},  \cite[\S 6, Ch. 3 ]{Fuchs+Salce-Moduovernon-doma:01} and \cite{McGovern-Bezoringwithalmo:08}.

\section{On the Kaplansky condition}\label{2}
By \cite{Gillman+Henriksen-Someremaabouelem:56}  a K-Hermite ring $R$ is an elementary divisor ring
iff it satisfies Kaplansky's condition (see \cite[Theorem 5.2]{Kaplansky-Elemdivimodu:49}):
{\em \begin{align*}\tag{K}
&\text{For any three elements $a,b,c$ in $R$ that generate the
ideal $R,$}\\
&\text{there exist elements $p,q\in R$ so that $(pa,pb+qc)=(1)$ in $R$.}
\end{align*}}

\begin{rem}\label{kaplsr1}
If $R$ is a ring of stable range $1$, then $R$ satisfies Kaplansky's condition with $p=1$. Thus, if $R$ is a ring satisfying Kaplansky's condition,
then $R$ is not necessarily K-Hermite, even if $R$ is a Noetherian local domain.
\end{rem}
Indeed, a local ring $R$ is of stable range $1$. If $R$ is also a Noetherian domain, then $R$ is K-Hermite iff $R$ is a principal ideal ring.
Hence a Noetherian local domain that is not a principal ideal ring is of stable range $1$, but it does not satisfy Kaplansky's condition.
\qed

In the proof of Lemma \ref{triangle} below, we will use the following well-known fact:
\begin{rem}\label{rows}
Let $R$ be a ring, let
$A$ be a matrix in $\M_{m,n}(R)$,  let $\mathbf r$ be a row in
in $\M_{1,n}(R)$, and let $1\le k\le n$. Then $\mathbf r$ belongs to the submodule of $R^n$ generated by the rows of the matrix $A$ iff
there exists a matrix $C\in \M_{k,m}(R)$ such that $\mathbf r$ is the first row of the matrix $CA$.
\end{rem}

\begin{lem}\label{triangle}
Let $A$   be a $2\times2$-matrix over a  ring $R$, and let $\mathbf
u$ be a unimodular row of length $2$ over $R$.

Then\\
 $\mathbf u$ belongs to the submodule of $R^2$ generated by the
rows of $A$ $\iff$ there exists an invertible matrix $P$ so that
$\mathbf u$ is the first row of  $PA$.
\end{lem}

\begin{proof}
\

$\implies$:

By Remark \ref{rows}, there exists a $2$-row $\mathbf r$ over $R$ so that
$\mathbf u=\mathbf r A$. Since the row $\mathbf u$ is unimodular, the row $\mathbf r$ is also unimodular.
 Since $\mathbf r$ is unimodular of length $2$,
there exists an invertible matrix $P$ with first row equal to
$\mathbf r$. Thus $\mathbf u$ is the first row of the matrix $PA$.

$\Longleftarrow$:

This follows from Remark \ref{rows}.
\end{proof}

\begin{lem}\label{general}
Let $A$ be a $2\times 2$ matrix over a ring $R$ so that its entries
generate the ideal $R$. Then\\
 $A$ is equivalent to a diagonal matrix
$\iff$\\
 the submodule of $R^2$ generated by the rows of $A$ contains
a unimodular row.

In this case $A$ is equivalent to a matrix of the form
$\bigl(\begin{smallmatrix}
1& 0\\
0&*\end{smallmatrix}\bigr)$.
\end{lem}

\begin{proof}
\

$\implies$

By assumption $A$ is equivalent to a diagonal matrix
$D=\bigl(\begin{smallmatrix}
d_1& 0\\
0&d_2\end{smallmatrix}\bigr)$, where $d_1,d_2\in R$. Thus $d_1,d_2$
generate the ideal $R$. The sum of the rows of $D$, namely,
$[d_1,d_2]$, is unimodular.

$\Longleftarrow$

By Lemma \ref{triangle}, the matrix $A$ is equivalent over $R$ to a
matrix $B$ with first row unimodular. Hence the submodule generated
by the columns of $B$ contains a column of the form
$\left(\begin{matrix} 1\\ *\end{matrix}\right)$. By Lemma
\ref{triangle} again (for columns), we obtain that $A$ is equivalent
to a matrix $\bigl(\begin{smallmatrix} 1& r\\
*&*\end{smallmatrix}\bigr)$. By subtracting the first column of the
matrix $\bigl(\begin{smallmatrix} 1& r\\
*&*\end{smallmatrix}\bigr)$  multiplied by $r$ from its second column and by a
similar elementary row transformation,  we obtain a diagonal matrix
of the form $\bigl(\begin{smallmatrix} 1&0\\
0&*\end{smallmatrix}\bigr)$. \end{proof}

\begin{thm}(see \cite[Theorem 5.2]{Kaplansky-Elemdivimodu:49} and
\cite[Corollary
5]{Gillman+Henriksen-Someremaabouelem:56}.)\label{kaplanskyc}

 Let $R$ be a  ring. Let
$A=\bigl(\begin{smallmatrix}
a& b\\
0&c\end{smallmatrix}\bigr)$ a triangular $2\times2$-matrix over $R$
so that $(a,b,c)=(1)$. Then $A$ is equivalent to a diagonal matrix
over $R$ iff there exist elements $p,q$ in $R$ so $(pa,pb+qc)=(1)$.
\end{thm}

\begin{proof}

Since $p[a,b]+q[0,c]=[pa,pb+qc]$ for any elements $p,q\in R$, the theorem follows from
Lemma \ref{general}.
\end{proof}

\begin{rem}\label{kaplanskyreformulated}
Let $R$ be any ring. If Kaplansky's condition $(pb+qc,pa)=(1)$ holds for elements $a,b,c,p,q\in R$, then
$$(pb+qc,a)=(p,c)=(1).$$
\end{rem}
Indeed, Kaplansky's condition implies that
$$(pb+qc,a)=(pb+qc,p)=(1),$$
 so $(p,c)=(1)$.
Cf. the next proposition.\qed

\begin{prop}\label{variantKaplansky}
Let $R$ be a K-Hermite ring, and let $a,b$ and $c$ be elements of
$R$ that generate the ideal $R$. Then the following four conditions are equivalent:
\begin{enumerate}
\item
The matrix $\bigl(\begin{smallmatrix}
a&b\\0&c\end{smallmatrix}\bigr)$ is equivalent to a diagonal matrix;

\item
There exist elements $p,q$ in $R$ so $(pa,pb+qc)=(1)$;

\item
There exist elements $p$ and $q$ in $R$ so that
$(pb+qc,a)=(p,c)=(1)$;

\item
For some elements $\lambda, u,v\in R$ we have $b+\lambda c=uv$, and
$(u,a)=(v,c)=(1)$.
\end{enumerate}
Moreover, in (4) we may choose the elements $u$ and $v$ such that $(u,v)=(1)$.
\end{prop}

\begin{proof}
\

$(1)\iff (2)$:

This follows from Theorem \ref{kaplanskyc}.

$(2)\implies (4)$:

Since $(pa, pb+qc)=(1)$ we obtain $(1)=(p,pb+qc)=(p,qc)$, so $(p,(pb+qc)c)=(1)$.  Let $v$ be an element of $R$ so that
$$vp\equiv 1\pmod {(pb+qc)c},$$ thus $vp\equiv 1\pmod c$. We  have
$v(pb+qc)\equiv b\pmod c$, so $v(pb+qc)=b+\lambda c$ for some element $\lambda\in R$. Hence
$b+\lambda c=uv$, where $u=pb+qc$, thus $(u,a)=(v,c)=(u,v)=(1)$.

$(4)\implies (3)$:

We have $b\equiv uv\pmod c$. Let $p\in R$ so that $pv\equiv 1\pmod c$. Hence $pb\equiv u\pmod c$, so there exists an element  $q\in R$ such that
$pb+qc=u$. Thus (3) holds.

$(3)\implies (2)$:

Since $R$ is a K-Hermite ring, we may write $(d)=(p,q)$ and $d=p_1p+q_1q$ with $(p_1,q_1)=(1)$. Hence $$(p_1,p_1b+q_1c)=(p_1,q_1c)=(1),$$ so
$(p_1a,p_1b+q_1c)=(p_1,c)=(1)$.  Condition (2) holds with $p$ and $q$ replaced by $p_1$ and $q_1$, respectively.
\end{proof}

In the proof of Proposition \ref{variantKaplansky}, we have used the assumption that $R$ is K-Hermite just for the implication
$(3)\implies (2)$.

\begin{rem}\label{bezoutcrit}
If $R$ is a B\'ezout domain, then the following condition is equivalent to the conditions of Proposition \ref{variantKaplansky}:

$(*)$ For some elements $\lambda, a_1,c_1\in R$ we have
$$b+\lambda c\mid (1-a_1a)(1-c_1c).$$
\end{rem}
Indeed, assume condition $(*)$. Let $u\in R$ so that
$$(u)=(b+\lambda c, 1-a_1a),$$
 thus $(u,a)=(1)$ and
$\frac {b+\lambda c} u~\mid~ \left(\frac {1-a_1a}u\right) (1-~c_1c)$. Since
$(\frac {b+\lambda c} u, \frac {1-a_1a}u)=(1)$, we see that
$v:=\frac {b+\lambda c}u$ divides $1-c_1c$, so $(v,c)=(1)$. Thus condition $(*)$ implies condition (4) of Proposition  \ref{variantKaplansky}.
The converse implication is obvious.\qed

Since a K-Hermite ring is an EDR iff each matrix of the form $\bigl(\begin{smallmatrix}a&b\\0&c\end{smallmatrix}\bigr)$ with $(a,b,c)=(1)$ has a diagonal reduction \cite{Kaplansky-Elemdivimodu:49}, Proposition \ref{variantKaplansky} provides necessary and sufficient conditions for a K-Hermite ring to be an EDR. We present an additional condition in the next proposition.

\begin{thm}\label{crit}
Let $R$ be a K-Hermite ring. The following two conditions are equivalent:
\begin{enumerate}
\item
$R$ is an elementary divisor  ring;
\item
For any elements $x,y,z\in R$ such that  $(x,y)=(1)$, there exists an element $\lambda\in R$ such that $x+\lambda y=uv$, where
$(u,z)=(v,1-z)=(1)$.
\end{enumerate}
Moreover, the elements $u$ and $v$ can be chosen such that $(u,v)=(1)$.
\end{thm}

\begin{proof}
\

$(1)\implies (2)$ [including the requirement that $(u,v)=(1)$]:

We apply condition (4) of Proposition \ref{variantKaplansky} to the elements $a=z,b=x,c=y(1-z)$.

$(2)\implies (1)$:

We verify condition (4) of Proposition \ref{variantKaplansky}. Let $(a,b,c)=(1)$. Let $(d)=(b,c)$, thus $(d,a)=(b,c,a)=(1)$. Hence
$a\mid 1-d_1d$ for some element $d_1\in R$.
Also $b=b_1d,c=c_1d$, where $(b_1,c_1)=(1)$.
We apply condition (2) of the present proposition to the elements
$$x=b_1,y=c_1,z=d_1d.$$ Thus there are elements $\lambda_1,u_1,v\in R$ so that $b_1+\lambda_1 c_1=u_1v$, where $(u_1,1-d_1d)=(v,d_1d)=(1)$. Let $u=du_1$, thus $(u,a)=1$. Let $\lambda=\lambda_1d$. Hence
$b+\lambda c=d(b_1+\lambda_1 c_1)=uv$, and $(u,a)=(1)$. We have $(v,c)=(v,dc_1)=(v,c_1)$ since $(v,d)=(1)$. Since $v$ divides $b_1+\lambda_1c_1$,
it follows that $(v,c_1)\mid b_1$, so $(v,c_1)=(1)$. Thus $(v,c)=(1)$, as required.
\end{proof}

\begin{prop}\label{critdomain}
Let $R$ be a B\'ezout domain. The following two conditions are equivalent:
\begin{enumerate}
\item
$R$ is an elementary divisor  ring;
\item
For any nonzero elements $x,y,z\in R$, there exists  elements $\lambda,a,b\in R$ such that $x+\lambda y\mid y(1-az)(1-b(1-z)$ in $R$.
\end{enumerate}
\end{prop}

\begin{proof}
$(1)\implies (2)$:

Let $(d)=(x,y)$, thus $\frac xd$ and $\frac yd$ are comaximal. By Theorem \ref{crit}, there are elements $\lambda,a,b\in R$ so that
$(\frac xd+\lambda \frac yd)\mid (1-az)(1-bz(1-z))$. Hence $x+\lambda y\mid d(1-az)(1-b(1-z))$, so $x+\lambda y\mid y(1-az)(1-b(1-z))$.

$(2)\implies (1)$:

Let $x_0$ and $y_0$ be comaximal elements in $R$, and let $z\in R$. Thus $(x_0+\lambda y_0)\mid y_0(1-az)(1-b(1-z))$ for some elements
$\lambda,a,b\in R$. Since the elements $x_0+\lambda y_0$ and $y_0$ are comaximal, we obtain that $(x_0+\lambda y_0)\mid (1-az)(1-b(1-z))$, so
$R$ is an EDR by Remark \ref{bezoutcrit}.
\end{proof}

\section{On rings of almost stable range $1$}\label{3}

\begin{prop}
Let $R$ be any ring. The following conditions are equivalent:
\begin{enumerate}
\item
$R$ is of almost stable range $1$;

\item
For each nonzero element $z\in R$, the ring $R/zR$ is of stable range $1$;

\item
For each three elements $x,y,z\in R$ so that $(x,y)=(1)$ and $z\ne0$, there exists an element $\lambda\in R$ so that $(x+\lambda y,z)=(1)$.
\end{enumerate}
\end{prop}

\begin{proof} Cf. \cite[Proposition 3.2, Ch. V]{Bass-Alge:68}.
\

$(1)\implies (2)\implies (3)$: Clear.

$(3)\implies (1)$:

Let $I$ be a nonzero ideal of $R$ and let $z$ be a nonzero element in $I$. Let $x+I,y+I$ be two  comaximal elements in $R/I$.
Hence there exist elements $r,s\in R$ such that $1-rx-sy\in I$. By assumption, there exists an element $\lambda\in R$ such that
$(x+\lambda(1-rx),z)R=R$. Thus $x+\lambda sy$ is invertible modulo the ideal $I$. It follows that $R/I$ is of stable range $1$, so $R$ is almost of stable range $1$.
\end{proof}

\begin{rem}\label{almoststable}
 The implication $(3)\implies (1)$ in the previous proposition is clear since if $T$ is a homomorphic image of a ring $R$ with finite stable range, then $\sr(T)\le \sr(R)$ \cite[Proposition 3.2, Ch. V]{Bass-Alge:68}, although this fact was not used explicitly, but rather its proof (in the above proof of the implication $(2)\implies (3)$
 we have $\sr(R/I)\le \sr(R/(z)=1$).
This fact implies that if $R$ is an arbitrary ring of stable range $1$, then $R$ is of almost stable range $1$, thus answering the question
in \cite[Remark 3.3]{McGovern-Bezoringwithalmo:08}. See also \cite[Proposition 3.2]{McGovern-Bezoringwithalmo:08}.
\end{rem}

As we have seen in \S\ref{2}, the stable range $1$ property implies Kaplansky's condition for an arbitrary ring. As for the converse,
even if $R$ is an elementary divisor domain, thus $R$ satisfies Kaplansky's condition, then $R$ does not necessarily has  even almost stable range $1$:

\begin{exmp}\label{warrenquest}
An elementary divisor domain (and so B\'ezout) that does not have almost stable range $1$ (this example answers the question in \cite[Remark 4.7]{McGovern-Bezoringwithalmo:08}).
\end{exmp}

We use a well-known example of a B\'ezout domain, namely, $R=\mathbb Z+X\mathbb Q[X]$ (for a general theorem on pullbacks of B\'ezout domains see
\cite[Theorem 1.9]{Houston+Taylor-Aritproppull:07}). $R$ is an elementary divisor ring by \cite[Theorem 4.61]{Costa+MottETAL-cons:78}. However, $R/X\mathbb Q[X]$ is isomorphic to $\mathbb Z$ and $\sr \mathbb Z=2$ (clearly, there is no integer $m$ such that $2+5m=\pm 1$. Thus $\sr \mathbb Z>1$. As mentioned in the introduction,
the stable range of any B\'ezout domain is $\le 2$, hence $\sr \mathbb Z=2$).
\qed

We conjecture that a B\'ezout domain that is a  pullback of type $\square$ (as defined in \cite{Houston+Taylor-Aritproppull:07})  of elementary divisor domains is again an EDR. In this case the conditions of
\cite[Theorem 1.9]{Houston+Taylor-Aritproppull:07} must be satisfied. If this conjecture proves to be false, thiwill yield a negative answer
to the question in \cite{Henriksen-Someremaelemdivi:55} whether a B\'ezout domain is an EDR.
\def\cprime{$'$} \def\cprime{$'$} \def\cprime{$'$} \def\cprime{$'$}
  \def\cprime{$'$}

\end{document}